\documentclass[12pt]{amsart}
 \usepackage[margin=1.6in]{geometry}
 \usepackage{amsmath,amssymb,amsthm,graphicx,amsxtra, setspace}
 \usepackage[utf8]{inputenc}
 \usepackage{mathrsfs}
 \usepackage{hyperref}
 \usepackage{xcolor}
 \usepackage{upgreek}
 \usepackage{mathtools}
 \allowdisplaybreaks
 
 \newtheorem{theorem}{Theorem}[section]

 \newtheorem{definition}{Definition}[section]

 \usepackage{latexsym}
 \usepackage{hyperref}
 \usepackage{graphicx}
 \usepackage{epstopdf}

\usepackage{amssymb,bbm,enumerate,bbm,amsmath}
\usepackage{changepage}
\usepackage[a4paper, left=1.3in, right=.3in]{}

\usepackage{mathrsfs}
\usepackage{tikz}
\usepackage{ dsfont}
\usepackage{hyperref}
\usetikzlibrary{arrows}
\usepackage{amsthm}
\usepackage{array,amsfonts}
\usepackage{amsmath}
\usepackage[utf8]{inputenc}
\usepackage[T1]{fontenc}
\usepackage{mathtools}
\usepackage{enumitem}
\usepackage[thinc]{esdiff}
\usepackage{multirow}
\usepackage{graphics}
\usepackage{amsmath,bm}

\newtheorem{remark}[theorem]{Remark}

\newtheorem{lemma}[theorem]{Lemma}

 \allowdisplaybreaks
 
 \let\originalleft\left
 \let\originalright\right
 \renewcommand{\left}{\mathopen{}\mathclose\bgroup\originalleft}
 \renewcommand{\right}{\aftergroup\egroup\originalright}

 \newcommand{\Addresses}{{
 		\footnote{

 			\noindent	 \textsuperscript{1,2} Department of Mathematics, Indian Institute of Technology Roorkee, Roorkee, 247667, India.	
 			
 			\noindent  \textit{e-mail\textsuperscript{1}:} \texttt{p\_yadav@ma.iitr.ac.in}
 			
 			\noindent  \textit{e-mail\textsuperscript{2}:} \texttt{tanuja.srivastava@ma.iitr.ac.in}

 }}}
 
\begin{document}
 	\title[]{The Maximum Likelihood Degree of Farlie-Gumbel-Morgenstern Bivariate Exponential Distribution\Addresses}
 	\author [Pooja Yadav.  Tanuja Srivastava]{Pooja Yadav\textsuperscript{1}.  Tanuja Srivastava\textsuperscript{2}}
 	\maketitle
    
 \begin{abstract}
		The maximum likelihood degree of a statistical model refers to the number of solutions, where the derivative of the log-likelihood function is zero, over the complex field. This paper examines the maximum likelihood degree of the parameter in Farlie-Gumbel-Morgenstern bivariate exponential distribution.
	\end{abstract}
    
	\section{Introduction}
	
	The probability density function (PDF) of Farlie-Gumbel-Morgenstern (FGM) or Gumbel's Type-II bivariate exponential distributed random vector $X=(x,y)^\top$, $(x, y)$ is in the first quadrant of $\mathbb{R}^2$ with association parameter $\theta$, is
	\begin{equation}
		\label{eq:1}
		f(x, y)=  e^{-(x+y)} \left[1 +\theta (2 e^{-x}-1)(2 e^{-y}-1)\right],
	\end{equation}
	where, $\theta \in [-1,1] \subset \mathbb{R}$, and the marginal distributions of each $x$ and $y$ are standard exponential. If $\theta =0 $, then $x$ and $y$ are mutually independent  \cite{KBJ}. 
	
	The FGM bivariate exponential distribution is commonly used in reliability, queueing theory, and actuarial science fields, and some of the applications of this distribution are given in  \cite{LB}. The simulation of random samples of the FGM bivariate exponential distribution is given in \cite{HNW}.
	
	Let $X_{1}=(x_1,y_1)^\top, X_{2}=(x_2,y_2)^\top, \ldots, X_{n}=(x_n,y_n)^\top$ be random sample from FGM bivariate exponential distribution, then the maximum likelihood estimator of $\theta$ is that value of $\theta \in [-1,1]$,  which maximizes the likelihood function given the data, if it exists.
	
	The likelihood function for $\theta$ is
	\begin{align*}
		L(\theta |X_1,X_2,\ldots X_n) &=\prod_{i=1}^{n} f(x_{i},y_{i})\\
		&= \prod_{i=1}^{n} e^{-(x_{i}+y_{i})} \left[1+\theta (2 e^{-x_i}-1)(2 e^{-y_i}-1) \right], 
	\end{align*}
	and, the log-likelihood function (up to an additive constant) is
	\begin{equation*}
		\ell(\theta)= \sum_{i=1}^{n} \log \left(1+\theta (2 e^{-x_i}-1)(2 e^{-y_i}-1)\right).
	\end{equation*}
	
	Therefore, the score equation of $\theta$ is 
	\begin{equation}
		\label{eq:2}
		\sum_{i=1}^{n} \frac{(2 e^{-x_i}-1)(2 e^{-y_i}-1)}{1+ \theta (2 e^{-x_i}-1)(2 e^{-y_i}-1)} =0.
	\end{equation}
	This equation is the summation of rational functions in $\theta$, which will have more than one solution and does not have a closed-form solution. So, it is necessary to apply some computational algebraic techniques to solve this.
	
	Since $\mathbb{R}$ is not an algebraically closed field, the solutions of the score equation are considered over the complex field $\mathbb{C}$.
	
	\section{The ML-degree of the association parameter of FGM bivariate exponential distribution}
	\begin{definition}[\textbf{Maximum likelihood degree}\cite{CHKS}]
		\label{def:1}
		The maximum likelihood degree or ML-degree of the association parameter $\theta$ of this model is the number of solutions of the score equation \eqref{eq:2}, counted with multiplicity over the complex field.
	\end{definition}
	For more details about the ML-degree, the readers can see \cite{CHKS}, \cite{DSS}.

	Let $\frac{1}{c_{i}}=(2 e^{-x_i}-1)(2 e^{-y_i}-1)$, for every $i=1,2,\ldots,n$. Then, the score equation \eqref{eq:2} can be rewritten as
	\begin{equation}
		\label{eq:3}
		\sum_{i=1}^{n} \frac{1}{(c_{i}+\theta) } =0,
	\end{equation}
	or
	\begin{equation*}
		\frac{h(\theta)}{k(\theta)} =0,
	\end{equation*}
	with 
	\begin{equation}
		\label{eq:4}
		h(\theta) = \displaystyle \sum_{i=1}^{n} \left( \displaystyle \prod_{j=1	j\ne i}^{n} \left( \theta +c_{j} \right) \right),
	\end{equation}
	and 
	\begin{equation}
		\label{eq:5}
		k(\theta)=\displaystyle \prod_{i=1}^{n} \left(\theta + c_{i} \right).
	\end{equation} 
	
	The solutions of the score equation are the zeros of $h(\theta)$. However, these solutions may contain the points where the score equation is not defined due to the cleared denominator. Therefore, the solutions of the score equation are the zeros of $h(\theta)$ that are not the zeros of $k(\theta)$. For the ML-degree, the common zeros of $h(\theta)$ and $k(\theta)$ should be removed from the solutions of $h(\theta)$.
	
	Since $h(\theta)$ is a polynomial of degree $n-1$, it will have $n-1$ zeros in the complex field, counted with multiplicity. The solutions of the score equation are in the variety of $h(\theta)$ (referred as $V(h)$). Hence, the ML-degree of $\theta \le (n-1)$. For the ML-degree of $\theta$, the points of concern are 
	\[V(h)\setminus \left(V(h)\cap V(k)\right) =V(h)\setminus V(h,k).\]	
	
	\begin{theorem}
		\label{thm:2}
		$V(h,k)\ne \emptyset$ if and only if there exists $l\ne m\in \{1,2,\ldots,n\}$ such that $c_{l} =c_{m}$.
	\end{theorem}
	
	\begin{proof}.
		Suppose $ h(\theta)$ and $k(\theta)$ have a common zero, say $\alpha $, that is,
		$h(\alpha)= 0$ and $k(\alpha)=0$.
		Consider
		\[k(\alpha)= \displaystyle \prod_{i=1}^{n} ( \alpha+c_{i}) =0,\]
		then $\exists$ some $l \in \{1,2,\ldots n\}$ such that $\alpha+c_{l}=0$ or $\alpha=-c_{l}$.
		
		Now, 
		\[h(\alpha)=  \displaystyle \sum_{i=1}^{n} \left(\displaystyle \prod_{j=1	j\ne i}^{n} (\alpha+c_{j}) \right) =\displaystyle \prod_{j=1	j\ne l}^{n} (\alpha+c_{j}),\]
		therefore, $h(\alpha)=0 \implies \alpha+c_{m}=0$, or $\alpha=-c_{m}, m\ne l.$
		
		Hence, $ c_{l}=c_{m}$.
		
		Conversely, suppose there exists $l\ne m\in \{1,2,\ldots,n\}$ such that $c_{l} =c_{m}$. Then, polynomials $\theta+c_{l} $ and $ \theta+c_{m}$, $l\ne m$ are same, and $\alpha_{1}= -c_{l}=-c_{m}$ is the zero of both polynomials $\theta+c_{l}$ and $\theta+c_{m} $.
		Now,
		\[k(\alpha_{1})= \displaystyle \prod_{i=1}^{n} ( \alpha_{1}+c_{i})=0,\]
		and
		\begin{align*}
			h(\alpha_{1})=  \displaystyle \sum_{i=1}^{n} \left(\displaystyle \prod_{j=1	j\ne i}^{n} (\alpha_{1}+c_{j}) \right),
		\end{align*}

  \begin{multline*}
    h(\alpha_{1})= \left( \displaystyle \prod_{j=1	j\ne l}^{n} (\alpha_{1}+c_{j}) \right)+ \left( \displaystyle \prod_{j=1	j\ne m}^{n} (\alpha_{1}+c_{j}) \right)\\
    +  \displaystyle \sum_{i=1, i\ne l,m}^{n} \left(\displaystyle \prod_{j=1	j\ne i}^{n} (\alpha_{1}+c_{j}) \right)=0.
   \end{multline*}
        
		Hence, $\alpha_{1}\in V(h,k)$ and $V(h,k)\ne \emptyset$.
	\end{proof}
	
	As mentioned above, the ML-degree of the association parameter $\theta$ of FGM bivariate exponential distribution is determined by the number of elements in $V(h)\setminus V(h,k)$, counted with multiplicity. Therefore, the number of elements in $V(h,k)$ is important for calculating the ML-degree. The possibility for $V(h,k)$ to be non-empty is discussed in the previous theorem. In the following lemma, the multiplicity of common zeros of $h(\theta)$ and $k(\theta)$ is counted in $h(\theta)$. 
	
	\begin{lemma}
		\label{thm:3}
		The multiplicity of a common zero of $h(\theta)$ and $k(\theta)$ in $h(\theta)$ is $n_{1}-1$, if exactly $n_{1}$ $(2\le n_{1} \le n)$ $c_{i}$'s are the same. 
	\end{lemma}
	\begin{proof}
		Given that there are exaclty $n_{1}$ $c_{i}$'s are the same, let denote them $c_{1},c_{2},\ldots,c_{n_{1}}$, that is, $c_{1}=c_{2}=\ldots=c_{n_{1}}$, then by \hyperref[thm:2]{theorem \ref{thm:2}}, $-c_{1}$ is a common zero of $h(\theta)$ and $k(\theta)$.
		Next, the multiplicity of $-c_{1}$ is counted in $h(\theta)$ as follows:
		\begin{align*}
			h(\theta)& = \displaystyle \sum_{i=1}^{n} \left( \displaystyle \prod_{j=1	j\ne i}^{n} \left( \theta +c_{j} \right) \right)\\
			& = \displaystyle \sum_{i=1}^{n_1} \left(\displaystyle \prod_{j=1	j\ne i}^{n} \left( \theta +c_{j} \right) \right)+\displaystyle \sum_{i=n_{1}}^{n} \left(  \displaystyle \prod_{j=1	j\ne i}^{n} \left( \theta +c_{j} \right) \right),
		\end{align*}
       
		or
		\begin{multline*}
			h(\theta)	=\left(\displaystyle \sum_{i=1}^{n_1} \left(\displaystyle \prod_{j=1	j\ne i}^{n_{1}} \left( \theta +c_{j} \right)\right)  \left(\displaystyle \prod_{j=n_{1}}^{n} \left( \theta +c_{j} \right) \right) \right)\\ +\left(\displaystyle \sum_{i=n_{1}}^{n} \left(  \displaystyle \prod_{j=1}^{n_{1}} \left( \theta +c_{j} \right) \right)  \left(  \displaystyle \prod_{j=n_{1}	j\ne i}^{n} \left( \theta +c_{j} \right) \right)\right),
		\end{multline*}
         \newpage
		or
		\begin{align*}
			h(\theta)&= n_{1} \left( \theta +c_{1} \right)^{n_{1}-1}  \left(\displaystyle \prod_{j=n_{1}}^{n} \left( \theta+c_{j} \right) \right)+ \left( \theta+c_{1} \right)^{n_{1}} \left(\displaystyle \sum_{i=n_{1}}^{n} \left( \displaystyle \prod_{j=n_{1}	j\ne i}^{n} \left( \theta +c_{j} \right) \right)\right)\\
			&=\left( \theta+c_{1}\right)^{n_{1}-1} h_{1}(\theta),
		\end{align*}
		where
		\begin{equation*}
			h_{1}(\theta)=n_{1} \left(\displaystyle \prod_{j=n_{1}}^{n} \left( \theta +c_{j}  \right) \right)+ \left( \theta+c_{1} \right) \left(\displaystyle \sum_{i=n_{1}}^{n} \left(  \displaystyle \prod_{j=n_{1}	j\ne i}^{n} \left( \theta +c_{j}  \right) \right)\right)
		\end{equation*}
		does not have $-c_{1}$ as a zero, since $c_{j}$ are distinct from $c_{1}$ $\forall j \in \{n_{1},\ldots,n\}$.
		
		Hence, the multiplicity of $-c_{1}$ in $h(\theta)$ is $n_{1}-1$.
	\end{proof}

	\begin{theorem}[\textbf{ML Degree}]
		\label{thm:4}
		Let \eqref{eq:3} have $p$ distinct $c_{i}$'s, and each $c_{i}$'s have multiplicity $n_{i}$ $(1\le n_{i} \le n)$. If there are exactly $l$ $n_{i}$'s $( >1)$ such that $\sum_{i=1}^{l}n_{i}=m$ $(\le n)$, then the ML-degree of the association parameter $\theta$ in FGM bivariate exponential distribution is $n+l-m-1$.
	\end{theorem}
	
	\begin{proof}
		Given that there are $p$ distinct $c_{i}$'s, say $c_{1},c_{2}, \ldots, c_{p}$ and say $c_{1},c_{2}, \ldots, c_{l}$ are repeated more than once, each with $n_{1},n_{2}, \ldots, n_{l}$ times, and $\sum_{i=1}^{l}n_{i}=m$. Since $c_{i}$'s are repeated, so by \hyperref[thm:2]{theorem \ref{thm:2}}, $V(h,k) \ne \emptyset$ and $V(h,k)=\{-c_{1},-c_{2},\ldots,-c_{l}\}$. Since, each $c_{j}$ is repeated with multiplicity $n_{j}\ge 2$, thus, by \hyperref[thm:3]{lemma \ref{thm:3}}, each $-c_{j}$ has multiplicity $n_{j}-1$ in $h(\theta)$ $\forall j \in\{1,2,\ldots,l\}$.
		
		The ML-degree of the association parameter $\theta$ in FGM bivariate exponential distribution is determined by the number of elements in $V(h)\setminus V(h,k)$ (counted with multiplicity), which is equal to
		\begin{equation*}
			(n-1)- \sum_{k=1}^{l}(n_{k}-1) = (n-1)-(m-l)=n+l-(m+1).
		\end{equation*}
		
		Hence, the ML-degree of the association parameter $\theta$ is $n+l-m-1$.
	\end{proof}
	
	\begin{remark}
		If $l=0$ in \hyperref[thm:4]{theorem \ref{thm:4}}, then $m=0$, that is, all $c_{i}$'s are distinct. In this case, $V(h,k)=\emptyset$, so the ML-degree of the association parameter $\theta$ is $n-1$. If $l=1$ and $m=n$ in \hyperref[thm:4]{theorem \ref{thm:4}}, that is, all $c_{i}$'s are the same, then the score equation \eqref{eq:3} will not be valid, hence this case is excluded from \hyperref[thm:4]{theorem \ref{thm:4}}. For this particular case, where all $c_{i}$'s are the same, the likelihood function is $L(\theta)=(1+\frac{\theta}{c_{1}})^{n}$, which is either an increasing or decreasing function according to $c_{1}$ being positive and negative (respectively). Thus, in this case, the MLE will be either $1$ or $-1$ since $-1\le \theta\le 1$. When $m$ is equal to $n$, then the value of $l$ is at least $2$. Hence, the ML-degree of the association parameter $\theta$ is greater than or equal to $1$. 
	\end{remark}

    \section*{Acknowledgements}The first author would like to thank the University Grants Commission, India, for providing financial support.


\begin{thebibliography}{20}
		\bibitem{CHKS}F. Catanese, S. Hoşten, A. Khetan and B. Sturmfels, The maximum likelihood degree, \emph{American Journal of Mathematics}, \textbf{128}(3), (2006) 671-697.
		\bibitem{DSS}M. Drton, B. Sturmfels, and S. Sullivant, \emph{Lectures on algebraic statistics} (Vol. 39), Springer Science and Business Media, (2008). 
    	\bibitem{HNW}Q. He, H. N. Nagaraja, and C. Wu, Efficient simulation of complete and censored samples from common bivariate exponential distributions, \emph{Computational Statistics}, \textbf{28}(6), (2013) 2479-2494.
	   \bibitem{KBJ}S. Kotz, N. Balakrishnan and N. L. Johnson,  \emph{Continuous multivariate distributions, Volume 1: Models and applications},  John Wiley and Sons, (Vol. 334) (2019).
	   \bibitem{LB}C. D. Lai and N. Balakrishnan, \emph{Continuous bivariate distributions}, Springer-Verlag New York, (2009).
	\end{thebibliography}
\end{document}